\documentclass[a4 paper,11pt]{amsart}
\usepackage{amsmath} %
\usepackage{amssymb} %
\usepackage{amscd} %
\usepackage{mathrsfs} %

\usepackage[dvipdfmx]{color} 
\usepackage[dvipdfmx]{graphicx} 
\usepackage{wrapfig} 
\usepackage{setspace} %

\usepackage[all]{xy} %
\usepackage{url} %

\usepackage[foot]{amsaddr} 

\usepackage{fancyhdr}

 \newtheorem{theorem}{Theorem}

 \newtheorem{proposition}[theorem]{Proposition}

  \newtheorem*{main theorem}{Main Theorem} 
  \newtheorem*{main a}{Theorem A}
 \newtheorem*{main b}{Theorem B}
  \newtheorem*{main c}{Theorem C}
   \newtheorem*{corollary*}{Corollary}
   
   \theoremstyle{definition}
   \newtheorem{definition}[theorem]{Definition}

\newcommand{\R}{\mathbb R}

\newcommand{\mP}[1]{Q({#1})}
\newcommand{\mmP}{Q}
\newcommand{\maximize}[1]{M_{{\rm max}}(#1)}

\begin{document}

\title{
{\large Ergodic Maximizing Measures of Non-Generic, Yet Dense Continuous Functions}
}

\author{Mao Shinoda}

\address{Department of Mathematics,
	Keio University, Yokohama,
	223-8522, JAPAN} 
\email{shinoda-mao@z3.keio.ac.jp}
\address{Phone: +81-45-566-1641+42706}
\address{Fax: +81-45-566-1642}


\begin{abstract}
	Ergodic optimization aims to single out dynamically invariant Borel probability measures which maximize the integral of a given ``performance" function.
	For a continuous self-map of a compact metric space and a dense set of continuous performance functions, 
	we show that the existence of uncountably many ergodic maximizing measures.
	We also show that, for a topologically mixing subshift of finite type and a dense set of continuous functions there exist uncountably many ergodic maximizing measures which are fully supported and have positive entropy.
\end{abstract}

\maketitle

\markboth{Mao Shinoda}{Ergodic Maximizing Measures of Non-Generic, Yet Dense Continuous Functions}
\section{introduction}

Ergodic optimization aims to single out dynamically invariant Borel probability measures which maximize the integral of a given ``performance" function.
It originates in variational methods in mechanical systems \cite{Mane, Mather} and 
several applications can be considered, for instance, in the thermodynamic formalism \cite{Lopes}, 
in chaos control \cite{Hunt Ott, Yuan Hunt}. 
(See bibliographical notes in \cite{Jenkinson E} for more details).

Consider a continuous self-map $T$ of a compact metric space $X$ and
a continuous performance function $\phi: X\rightarrow \R$.
A $T$-invariant Borel probability measure $\mu$ is called a {\it $\phi$-maximizing measure} if 
\begin{align*}
	\max_{\nu\in M(X,T)}\int \phi\ d\nu=\int \phi\ d\mu,
\end{align*}
where $M(X,T)$ denotes the space of $T$-invariant Borel probability measures endowed with the weak*-topology.
We investigate properties of $\phi$-maximizing measures
by considering the set $\maximize{\phi}$ of all $\phi$-maximizing measures.
A performance function $\phi$ is {\it uniquely maximized} if $\maximize{\phi}$ is a singleton.
Jenkinson shows that 
a generic continuous function is uniquely maximized \cite{Jenkinson E}.
If $T$ has the specification property, 
the unique maximizing measure of a generic continuous function is fully supported and has zero entropy
\cite{Bousch Jenkinson, Bremont, Jenkinson E, Morris}.
However, it is difficult  to tell whether or not a given performance function is generic.
Indeed, no concrete example of a continuous function is known which 
is uniquely maximized by a fully supported measure \cite[Problems 3.9 and 4.3]{Jenkinson E}.
Hence it is natural to investigate properties which are non-generic, yet hold for a reasonably large set of functions.

In this paper we pay attention to non-generic properties of maximizing measures.
Denote by $C(X)$ the space of continuous functions endowed with the supremum norm
and by $M_e(X,T)$ the set of ergodic elements of $M(X,T)$.
Recall that $M_e(X,T)$ is the set of extrema points of $M(X,T)$.
Let us say that $M_e(X,T)$ is arcwise-connected if
for every $\mu$ and $\nu\in M_e(X,T)$ there exists a homeomorphism $[0,1]\ni t\mapsto f_t \in M_e(X,T)$
on its image such that $f_0=\mu$ and $f_1=\nu$.
Note that if $M_e(X,T)$ is arcwise-connected then $M(X,T)$ is not a singleton.
We prove that the set of uncountably maximized continuous functions is dense in $C(X)$, provided $M_e(X,T)$ is arcwise-connected.

\begin{main a}
	Let $T$ be a continuous self-map of a compact metric space $X$.
	Suppose $M_e(X,T)$ is arcwise-connected.
	There exists a dense subset $\mathcal{D}$ of $C(X)$ such that
	for every $\phi$ in $\mathcal{D}$ the set $\maximize{\phi}$ contains uncountably many ergodic elements.
\end{main a}
Examples to which Theorem A applies include topologically mixing subshifts of finite type and Axiom A diffeomorphisms.
The space of invariant measures of these systems is the Poulsen simplex \cite{Sigmund1}:
the infinite simplex for which the set of its extremal points is dense.
The denseness of the set of ergodic measures implies its arcwies-connectedness
because the Poulsen simplex and the set of its extremal points are homeomorphic to the Hilbert cube $[0,1]^\infty$ and its interior $(0,1)^\infty$ respectively \cite{Gelfert}.
Sigmund shows that
the specification, which the above examples actually have,  implies the denseness of the set of ergodic elements \cite{Sigmund2}.
Several extensions of Sigmund's result have been considered under some generalized versions of the specification (see \cite{Gelfert}).
In one-dimensional case,
Blokh shows that continuous topologically mixing interval maps have the specification \cite{Blokh}
and a discontinuous version is studied in \cite{Buzzi}.

The arcwise-connectedness of the set of ergodic measures is strictly weaker than the denseness of it.
For example the set of ergodic measures of the Dyck shift \cite{Krieger} is not dense but arcwise-connected.
The connectedness of the set of ergodic measures for some partially hyperbolic systems is studied in \cite{Gorodetski}.
On the other hand, there do exist systems for which $M_e(X,T)$ is not arcwise-connected: 
Cortez and Rivera-Letelier show that for the restriction of some logistic maps  to the omega limit set of the critical points, the sets of ergodic measures become totally disconnected \cite{Cortez}.

An idea of our proof of Theorem A is to perturb a given continuous function $\phi_0$
to create another $\phi$ 
so that the function $\mu\mapsto \int \phi\ d\mu$ defined on an arc of ergodic measures has a ``flat" part (see FIGURE \ref{BP}).
The Bishop Phelps theorem allows us to construct such a perturbation.
In order to use the Bishop Phelps theorem, we use the fact that maximizing measures are characterized as 
``tangent measures" to the convex functional
\begin{align*}
Q: C(X)\ni\phi\mapsto \max_{\nu\in M(X,T)}\int \phi\ d\nu\in \R.
\end{align*}
See Proposition \ref{maximizing measure}.
The use of the Bishop Phelps theorem has been inspired by \cite{Feliks} (see also \cite{Israel}).

It is worthwhile to remark that our perturbation does not work in the Lipschitz topoligy.
In the course of the proof of Theorem A we show the following statement.
\begin{corollary*}
Let $\phi_0\in C(X)$ and $\mu\in \maximize{\phi_0}$.
For any neighborhood $U$ of $\phi_0$ and any open neighborhood $V$ of $\mu$
there exists $\phi$ in $U$ such that
$V\cap \maximize{\phi}$ contains uncountably many ergodic elements.
\end{corollary*}

On the other hand, in the space of Lipschitz continuous functions
a phenomenon called a  ``lock up on periodic orbits" occurs:
for a Lipschitz continuous functions which is uniquely maximized by a periodic measure
one cannot realize a perturbation breaking the uniqueness of maximizing measure \cite{Yuan Hunt} (See also \cite{Lopes}). 

For the subshift of finite type,
one can choose the arc of ergodic measures used in the proof of Theorem A
from fully supported measures with positive entropy.
Hence, slightly modifying the proof of Theorem A we obtain the next theorem.

\begin{main b}
	Let $(X,T)$ be a topologically mixing subshift of finite type.
	There exists a dense subset $\mathcal{D}$ of $C(X)$ such that
	for every $\phi$ in $\mathcal{D}$ the set $\maximize{\phi}$ contains uncountably many ergodic elements
	which are fully supported and positive entropy.
\end{main b}

The rest of this paper is organized as follows. 
In Section \ref{Preliminary} we collect preliminary results on functional analysis and invariant Borel probability measures.
In Section \ref{proof of main} the theorems are proved.

 \begin{figure}[htbp]
 \begin{center}
  \includegraphics[width=150pt]{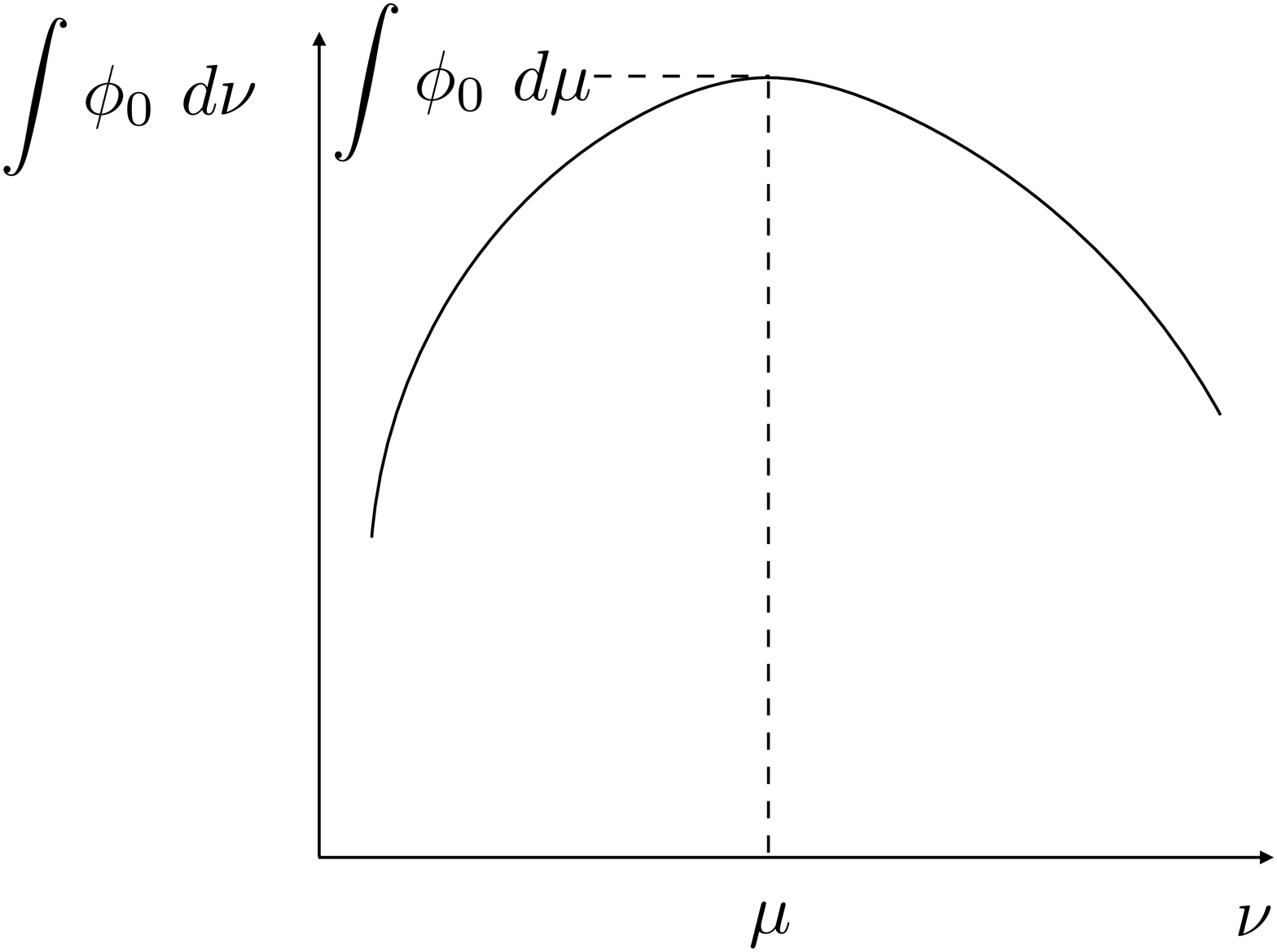}
  \raise 2cm\hbox{$\underset{\mbox{perturb}}{\longrightarrow}$}
   \includegraphics[width=150pt]{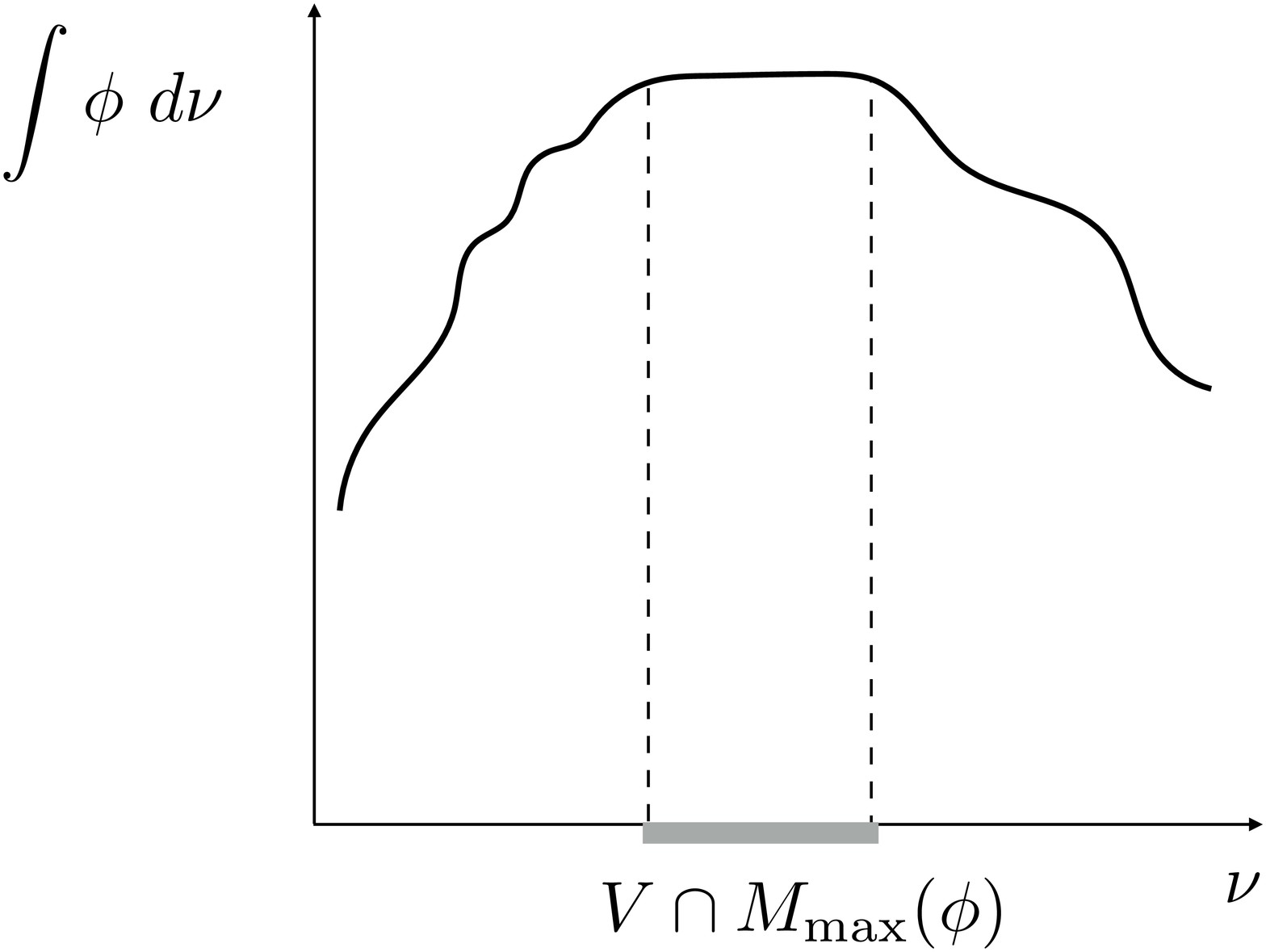}
 \caption{A schematic picture of the perturbation: a given function $\phi_0$ (left); the perturbed one $\phi$ (right).} \label{BP}
 \end{center}
\end{figure}

 \section{Preliminaries} \label{Preliminary}
 \subsection{Bishop Phelps Theorem}
First we see the Bishop Phelps theorem, which is concerned with a convex functional on a Banach space.
We begin with definitions of basic notions.

 \begin{definition} \label{tangent bdd}
 	A functional $\Gamma : V\rightarrow \R$ on a Banach space $V$ is {\it convex} if 
	$$\Gamma(t\phi+(1-t)\psi)\leq t\Gamma(\phi)+(1-t)\Gamma(\psi)$$ for all $\phi, \psi\in V$ and $t\in [0,1]$.
	
 	 Let $\Gamma : V\rightarrow \R$ be a convex and continuous functional on a Banach space $V$.
	 A bounded linear functional $F$ is {\it tangent} to $\Gamma$ at $\phi\in V$ if
	 \begin{align*}
	 	F(\psi)\leq\Gamma(\phi+\psi)-\Gamma(\phi)
	 \end{align*}
	 for all $\psi\in V$.
	 
	 A bounded linear functional $F$ is {\it bounded} by $\Gamma$ if
	$ 
	 	F(\psi)\leq\Gamma(\psi)
	 $ 
	 for all $\psi\in V$.
 \end{definition}
For a bounded linear functional $F$ on a Banach space $V$ define 
\begin{align*}
	\|F\|=\sup\{|F(\phi)| : \phi\in V\ \mbox{with}\ \|\phi\|=1\}. 
\end{align*}
This becomes a norm of the set $V^*$ of all bounded linear functionals on $V$.
Note that $V^*$ becomes a Banach space with this norm.
  The Bishop Phelps theorem states that 
$\Gamma$-bounded functionals can be approximated by $\Gamma$-tangent ones
with respect to this norm.

 \begin{theorem}\cite[Theorem V.1.1.]{Israel} \label{BP theorem}
 Let $\Gamma : V\rightarrow \R$ be a convex and continuous functional on a Banach space $V$.
  For every bounded linear functional $F_0$ bounded by $\Gamma$, $\phi_0\in V$ and $\varepsilon>0$,
there exist a bounded linear functional $F$ and $\phi\in V$ such that $F$ is tangent to $\Gamma$ at $\phi$ and 
\begin{align*}
	\|F_0-F\|\leq\varepsilon\quad \mbox{and} \quad \|\phi_0-\phi\|\leq\frac{1}{\varepsilon}(\Gamma(\phi_0)-F_0(\phi_0)+s(F_0)),
\end{align*}
where $s(F_0)=\sup\{F_0(\psi)-\Gamma(\psi): \psi\in V\}$.
\end{theorem}

\subsection{The Space of Borel Probability Measures}
Denote by $M(Y)$ the set of all Borel probability measures on a compact metric space $Y$.
In our setting, $M(X,T)$ is compact and metrizable in the weak*-topology.
Then we will also consider $M(M(X,T))$ in the following section.
Denote by $C(Y)$ the Banach space of all continuous functions on $Y$ with the supremum norm.
By the Riesz representation theorem, the set
\begin{align*}
	\left\{F \in C(Y)^* : F\ \mbox{is positive and normalized} \right\}
\end{align*}
can be identified with $M(Y)$.
With this identification,
the norm and the notion of being tangent and being bounded in Definition \ref{tangent bdd} carry over to elements of $M(Y)$.
For $\mu\in M(X,T)$ we denote by $\mu(\phi)$ the integral of a continuous function $\phi$ by $\mu$.

\subsection{Ergodic Decomposition} \label{Ergodic Decomposition}

For a $T$-invariant measure $\mu$ 
there exists a unique Borel probability measure $\alpha$ on $M(X,T)$ 
such that $\alpha(M(X,T)\setminus M_e(X,T))=0$ and 
\begin{align*}
	\mu(\phi)=\int_{M_e(X,T)} m(\phi)\ d\alpha(m) 
\end{align*}
for all $\phi\in C(X)$.
We call $\alpha$ the ergodic decomposition of $\mu$.
Note that $M(X,T)$ is a nonempty compact convex set
and $M_e(X,T)$ coincides with the set of its extremal points.
Since the ergodic decomposition of a $T$-invariant measure is unique, $M(X,T)$ is a Choquet simplex.
From the Theory of a Choquet simplex,
for the ergodic decompositions $\alpha_1, \alpha_2$ of  $\mu_1, \mu_2\in M(X,T)$,
 we have $\|\mu-\nu\|=\|\alpha_\mu-\alpha_\nu\|$.
 See \cite[Appendix A.5]{Ruelle} and the references therein.
 
\section{Proofs of the theorems} \label{proof of main}
\subsection{On the proof of Theorem A} \label{proof of main a}
Define a functional $\mmP : C(X)\rightarrow \R$ by
 \begin{align*}
 	\mP{\phi}=\max\left\{\mu(\phi) : \mu\in M(X,T)\right\}.
 \end{align*}
 Note that $\mmP$ is continuous and convex.
Maximizing measures are characterized by tangency to $Q$.
%

\begin{proposition}\cite[Lemma 2.3]{Bremont} \label{maximizing measure}
Let $T$ be a continuous self-map of a compact metric space $X$
and $\phi \in C(X)$. 
Then $\mu\in M(X,T)$ is tangent to $\mmP$ at $\phi$
if and only if 
$\mu\in \maximize{\phi}$.
\end{proposition}

First we consider the ergodic decomposition of a $\phi$-maximizing measure.
The following proposition states that invariant measures in the support of the ergodic decomposition of a $\phi$-maximizing measure are also $\phi$-maximizing.
The support of an ergodic decomposition $\alpha$ is defined by 
 $\mbox{supp}(\alpha)=\bigcap C$ where the intersection is taken over all closed subsets $C$ of $M(X,T)$ with $\alpha(C)=1$.
 Note that $\alpha({\rm supp}(\alpha))=1$, since $M(X,T)$ has a countable basis.

\begin{proposition} \label{ergodic decomp supp}

Let $T$ be a continuous self-map of a compact metric space $X$ and $\phi\in C(X)$.
Let $\mu\in \maximize{\phi}$ and let $\alpha$ be the ergodic decomposition of $\mu$.
Then ${\rm supp}(\alpha)$ is contained in $\maximize{\phi}$.
\end{proposition}

\begin{proof}

	Let $N=\{\nu \in M(X,T) : \int \phi \ d\nu<\int \phi \ d\mu\}$.
	Suppose $\alpha(N)>0$. 
	Then
	\begin{align*}
		\mu(\phi)
			&=\int_{M(X,T)} m(\phi) \ d\alpha(m)\\
			&=\int_{N} m(\phi) \ d\alpha(m)+\int_{M(X,T)\setminus N} m(\phi) \ d\alpha(m)\\
			&<\alpha(N)\mu(\phi)+\alpha(M(X,T)\setminus N)\mu(\phi)\\
			&=\mu(\phi).
	\end{align*}
	This is a contradiction and then we have $\alpha(N)=0$. 
	Since $M(X,T)\setminus N=\maximize{\phi}$ and $\maximize{\phi}$ is closed, we have ${\rm supp}(\alpha) \subset \maximize{\phi}$.
\end{proof}

Second we construct a non-atomic Borel probability measure $\alpha$ on $M(X,T)$ supported in $M_e(X,T)$ for a given $\phi\in C(X)$.
The point of the construction is this ${\rm supp}(\alpha)$ gives positive weight to the set of
ergodic measures
for which the integrals of $\phi$ are $\varepsilon$-close to the maximum value $Q(\phi)$.
The arcwise-connectedness of $M_e(X,T)$ is essential for the following construction.

\begin{proposition} \label{non-atomicmeasure}

	Let $T$ be a continuous self-map of a compact metric space $X$.
	Suppose $M_e(X,T)$ is arcwise-connected.
	Then for every $\phi\in C(X)$
	there is a non-atomic Borel probability measure $\alpha$ on $M(X,T)$ such that
	${\rm supp}(\alpha)\subset M_e(X,T)$ and 
	for all $\varepsilon>0$,
	\begin{align*}
		\alpha\left(\left\{\mu \in M_e(X,T): \mP{\phi}-\varepsilon\leq\mu(\phi)\right\}\right)>0.
	\end{align*}
\end{proposition}

\begin{proof}

	Pick $\phi \in C(X)$ and let $\mu$ be a $\phi$-maximizing measure.
	Pick $\nu\in M_e(X,T)\setminus\{\mu\}$.
	By the assumption, there exists an arc $f$ from $\mu$ to $\nu$.
	Let ${\rm Leb}_{[0,1]}$ denote the Lebesgue measure on $[0,1]$ and $\alpha=f_{*}{\rm Leb}_{[0,1]}$. 
	Since $f$ is a homeomorphism, the inverse image of a point in $f([0,1])$ is a singleton.
	Hence $\alpha$ is non-atomic.
	Since $t\in [0,1]\mapsto \int \phi \ df_t$ is continuous, the set
	$\{t\in [0,1] : \mP{\phi}-\varepsilon\leq\int \phi\ df_t\}$
	has nonempty interior.
	Since the image of the set by $f$ is contained in $M_e(X,T)$,
	$\alpha$ is supported in $M_e(X,T)$ and satisfies the desired inequality.
\end{proof}

We now prove Theorem A.
\begin{proof}[Proof of Theorem A]

	Note that every Borel probability measure $\mu$ is bounded by $Q$.
	Pick $\phi_0\in C(X)$ and $0<\varepsilon<\frac{1}{2}$.  
	Let $\tilde{\alpha}$ be a non-atomic Borel probability measure for which the conclusion of Proposition \ref{non-atomicmeasure} holds with $\phi=\phi_0$.
	Put $\delta=\varepsilon^2$ and
	\begin{align*}
	A_\delta=\{\mu \in M_e(X,T) : \mP{\phi_0}-\delta\leq\mu(\phi_0 )\}.
	\end{align*}
	Denote by $\alpha_0$ the conditinal measure of $\alpha_0$ on $A_\delta$, namely
	\begin{align*}
		\alpha_0(B)=\frac{1}{\tilde{\alpha}(A_\delta)}\tilde{\alpha}(A_\delta\cap B)
	\end{align*}
	for every Borel subsets $B$ of $M(X,T)$.
	Since $\tilde{\alpha}(A_\delta)>0$ by Proposition \ref{non-atomicmeasure}, $\alpha_0$ is well-defined.
	Note that $\alpha_0$ is also supported in $M_e(X,T)$.
	
	Let $\mu_0=\int_{M_e(X,T)}m \ d\alpha_0(m)$.
	By Theorem \ref{BP theorem} applied to $\phi_0$, $\mu_0$ and $\varepsilon$,
	 there exist $\phi\in C(X)$ and $\mu\in M(X)$ such that $\mu$ is tangent to $\mmP$ at $\phi$,
	 $\|\mu-\mu_0\|\leq\varepsilon$ and 
	 \begin{align*}
	 \|\phi-\phi_0\|
	 	\leq\frac{1}{\varepsilon}\left(\mP{\phi_0}-\mu_0(\phi_0 )\right)
		\leq\frac{1}{\varepsilon}\delta=\varepsilon.
	 \end{align*}
	 Then $\phi$ is $\varepsilon$-close to $\phi_0$
	 and by Proposition \ref{maximizing measure} $\mu$ is a $\phi$-maximizing measure.
	 
	 Next we show the existence of uncountably many ergodic $\phi$-maximizing measures.
	Let $\alpha$ be the ergodic decomposition of $\mu$
	and we have
	\begin{align}
	\|\alpha-\alpha_0\|=\|\mu-\mu_0\|\leq\varepsilon.\label{ergodic decomp iso 2}
	\end{align}
	Let $\rho>\frac{1-\varepsilon}{2}>0$.
	Since $\alpha_0$ is a Borel probability measure and ${\rm supp}(\alpha)$ is a closed set, there is an open set $U$ such that
	${\rm supp}(\alpha)\subset U$ and $\alpha_0(U\setminus {\rm supp}(\alpha))<\rho$.
	Since $M(X,T)$ is a metric space, there is a continuous function $g : M(X,T)\rightarrow [0,1]$
	which vanishes on $M(X,T)\setminus U$ and is identically $1$ on ${\rm supp}(\alpha)$.
	Hence we have
	\begin{align*}
		\alpha_0({\rm supp}(\alpha))
			>\alpha_0(U)-\rho
			\geq \int g\ d\alpha_0-\rho.
	\end{align*}
	The inequality in \eqref{ergodic decomp iso 2} implies 
	\begin{align*}
		-\varepsilon\leq\int h\ d\alpha -\int h\ d\alpha_0\leq\varepsilon
	\end{align*}
	for all $h \in C(M(X,T))$ with $\|h\|=1$.
	Hence we have
	\begin{align}
		\alpha_0({\rm supp}(\alpha))	\label{support ineq}
			&>\int  g\ d\alpha_0-\rho\\\nonumber
			&\geq \int g\ d\alpha-\rho-\varepsilon\\\nonumber
			&\geq \alpha({\rm supp}(\alpha))-\rho-\varepsilon
			=1-\rho-\varepsilon>0. \nonumber
	\end{align}

	Since $\alpha_0$ is non-atomic and supported in $M_e(X,T)$,
	${\rm supp}(\alpha)$ contains uncountably many ergodic elements.
	By Proposition \ref{ergodic decomp supp} 
	we have ${\rm supp}(\alpha)\subset \maximize{\phi}$,
	and the proof is complete.
\end{proof}

\subsection{On the proof of Theorem B}\label{proof of main b}

The following result by Sigmund is essential for our proof of Theorem B.
For $\mu,\nu\in M_e(X,T)$ a continuous function $[0,1]\in t \mapsto f_t\ni M_e(X,T)$ which satisfies $f_0=\mu$ and $f_1=\nu$ is called a {\it path} from $\mu$ to $\nu$.

\begin{theorem}\cite{Sigmund3}\label{construction of Sigmund}
Let $(X,T)$ be a topologically mixing subshift of finite type.
Then for every $\mu, \nu\in M_e(X,T)$ there exists a path $f$ from $\mu$ to $\nu$ with the following properties:
	(i) for every measure $m\in f([0,1])$, $f^{-1}(\{m\})$ is a countable set;
	(ii) every measure $m\in f([0,1])$ except for countably many ones is fully supported and has positive entropy. 
\end{theorem}

\begin{proof}[Proof of Theorem B]

Pick $\phi_0\in C(X)$ and $0<\varepsilon<\frac{1}{2}$.
We obtain a non-atomic Borel probability measure on $M_e(X,T)$ by modifying the proof of Proposition \ref{non-atomicmeasure}.
Let $\mu$ be a $\phi_0$-maximizing measure and pick $\nu\in M_e(X,T)\setminus\{\mu\}$.
Let $f$ be a path from $\mu$ to $\nu$ for which the conclusion of Theorem \ref{construction of Sigmund} holds.
Since the inverse image of any point is countable, 
$\tilde{\alpha}=f_*{\rm Leb}_{[0,1]}$ becomes a non-atomic Borel probability measure supported in $M_e(X,T)$.
Then for $\tilde{\alpha}$ the inequality in Proposition \ref{non-atomicmeasure} holds with $\phi_0$ and $\varepsilon^2$.

Following the proof of Theorem A,
we define $\alpha_0$ to be the restriction of $\tilde{\alpha}$ to
the set $A_{\varepsilon^2}=\{\mu \in M_e(X,T) : \mP{\phi_0}-\varepsilon^2\leq\mu(\phi_0 )\}$ 
and obtain $\phi\in C(X)$
and a Borel probability measure $\alpha$
such that $\|\phi_0-\phi\|\leq\varepsilon$ and ${\rm supp}(\alpha)\subset\maximize{\phi}$.

By Theorem \ref{construction of Sigmund}, ${\rm supp}(\tilde{\alpha})=f([0,1])$ contains uncountably many ergodic elements
which are fully supported and have positive entropy.
The definition of $\alpha_0$ implies
 \begin{align*}
 {\rm supp}(\alpha_0)={\rm supp}(\tilde{\alpha})\cap A_{\varepsilon^2}.
 \end{align*}
and ${\rm supp}(\alpha_0)$ still contains uncountably many ergodic elements
which are fully supported and have positive entropy.
By \eqref{support ineq} in the proof of Theorem A we have
\begin{align*}
	\alpha_0({\rm supp}(\alpha_0)\cap {\rm supp}(\alpha))=\alpha_0({\rm supp}(\alpha))>0.
\end{align*}
Since $\alpha_0$ is non-atomic and supported in $M_e(X,T)$, this implies ${\rm supp}(\alpha_0) \ \cap\  {\rm supp}(\alpha)$ contains uncountably many ergodic elements
which are fully supported and have positive entropy.
By Proposition \ref{ergodic decomp supp} we have
\begin{align*}
	{\rm supp}(\alpha_0)\cap {\rm supp}(\alpha)\subset {\rm supp}(\alpha)\subset  \maximize{\phi}
\end{align*}	
 and the proof is complete.
\end{proof}
\subsection*{Acknowledgments}
This research is partially supported by the JSPS Core-to-Core Program ``Foundation of a Global Research Cooperation Center in Mathematics focused on Number Theory and Geometry".
%

%
\end{document}